\documentclass[10pt]{amsart}
\usepackage{amsthm}
\usepackage{graphicx}
\usepackage{amssymb}
\usepackage{color}
\usepackage{tikz}
\usepackage{amsmath}
\usepackage{hyperref}
\usepackage{float}
\usepackage[T1]{fontenc}
\usepackage[utf8]{inputenc}

\numberwithin{equation}{section}

\newtheorem{theorem}{Theorem}[section]
\newtheorem{definition}{Definition}[section]

\newtheorem{lemma}[theorem]{Lemma}
\newtheorem{conjecture}[theorem]{Conjecture}
\newtheorem{proposition}[theorem]{Proposition}
\newtheorem{corollary}[theorem]{Corollary}
\newtheorem{remark}[theorem]{Remark}

\def\bclaim{\begin{claim}}
\def\eclaim{\end{claim}}
\def\bdefin{\begin{definition}}
\def\edefin{\end{definition}}
\def\bcor{\begin{corollary}}
\def\ecor{\end{corollary}}
\def\bthm{\begin{theorem}}
\def\ethm{\end{theorem}}
\def\bconj{\begin{conjecture}}
\def\econj{\end{conjecture}}
\def\blem{\begin{lemma}}
\def\elem{\end{lemma}}
\def\blemma{\begin{lemma}}
\def\elemma{\end{lemma}}
\def\bprop{\begin{proposition}}
\def\eprop{\end{proposition}}
\def\bremark{\begin{remark}}
\def\eremark{\end{remark}}

\newcommand{\PP}{{\mathbb P}} \newcommand{\RR}{\mathbb{R}}
\newcommand{\QQ}{\mathbb{Q}} 
\newcommand{\ZZ}{{\mathbb Z}}

\def\o{\omega}

\newcommand{\dbar}{\bar\partial}
\newcommand{\ddbar}{\partial\dbar}

\newcommand{\vol}{{\operatorname{Vol}}}

\def\max{{\operatorname{max}}}

\def\Ric{\hbox{\rm Ric}\,}

\def\beq{\begin{equation}}
\def\eeq{\end{equation}}
\def\beqno{\begin{equation*}}
\def\eeqno{\end{equation*}}
\def\eaeq{\end{aligned}}
\def\baeq{\begin{aligned}}

\def\bpf{\begin{proof}}
\def\epf{\end{proof}}

\def\eaeq{\end{aligned}}
\def\baeq{\begin{aligned}}

\chardef\inodot="10

\newcommand{\bQ}{\mathbb{Q}}
\newcommand{\bP}{\mathbb{P}}

\newcommand{\bZ}{\mathbb{Z}}

\newcommand{\bR}{\mathbb{R}}

\newcommand{\ord}{\mathrm{ord}}

\newcommand{\Vol}{\mathrm{Vol}}
\newcommand{\cX}{\mathcal{X}}
\newcommand{\Pic}{\mathrm{Pic}}
\newcommand{\cL}{\mathcal{L}}

\title[Optimal volume upper bound for K\"ahler manifolds]{
On the optimal volume upper bound for K\"ahler manifolds with positive Ricci curvature}

\usepackage{hyperref}
\hypersetup{colorlinks,linkcolor={blue},citecolor={red}}

\begin{document}

\author[Kewei Zhang]{Kewei Zhang \\(W\MakeLowercase{ith an} A\MakeLowercase{ppendix by} Y\MakeLowercase{uchen} L\MakeLowercase{iu})}
\address{Beijing International Center for Mathematical Research, Peking University.}
\email{kwzhang@pku.edu.cn}

\maketitle

\begin{abstract}
Using $\delta$-invariants and Newton--Okounkov bodies, we derive the optimal volume upper bound for K\"ahler manifolds with positive Ricci curvature, from which we get a new characterization of the complex projective space.
\end{abstract}
\maketitle

\section{Introduction}
\label{sec:intro}

Let $(X,g)$ be an $m$-dimensional Riemannian manifold such that
$$
\Ric(g)\geq (m-1)g.
$$
Then the well-known Bishop--Gromov volume comparison says that
$$
\vol(X,g)\leq\vol(S^{m}, g_{_{S^{m}}}),
$$
and the equality holds if and only if $(X,g)$ is isometric to the standard $m$-sphere $S^m$. However, suppose in addition that $X$ has a complex structure $J$ such that $(X,g,J)$ is K\"ahler, then G. Liu \cite{L14} shows that this volume upper bound is \emph{never sharp} (unless $X=\PP^1$), in the sense that there exists a dimensional gap $\epsilon(n)>0$ such that
$$
\vol(X,g)\leq\vol(S^{m}, g_{_{S^{m}}})-\varepsilon(n).
$$
This distinguishes the K\"ahler geometry from the Riemannian case. So it is natural to ask what the optimal volume upper bound in the K\"ahler setting should be. A folklore conjecture predicts that, in the K\"ahler setting, the complex projective space equipped with the Fubini--Study metric should attain the maximal volume. Regarding this problem, a significant progress was made by K. Fujita \cite{F}, who gave an affirmative answer in the case of K\"ahler--Einstein manifolds, and whose approach is, interestingly enough, purely algebraic. K. Fujita's breakthrough was also mentioned by S.K. Donaldson in his 2018 ICM talk (cf. \cite{Don18}). However, an answer for general K\"ahler manifolds with positive Ricci curvature is still missing.

The purpose of this paper is to completely solve this problem by using some key input from algebraic geometry and convex geometry.
Our main result is stated as follows.

\begin{theorem}
\label{thm:vol-kahler}
Let $(X,\o)$ be an $n$-dimensional K\"ahler manifold with
$
\Ric(\o)\geq(n+1)\o.
$
Then one has
$
\vol(X,\omega)\leq\vol(\PP^n,\omega_{FS}),
$
and the equality holds if and only if $(X,\o)$ is biholomorphically isometric to $(\PP^n,\o_{FS})$. Here $\o_{FS}$ denotes the Fubini--Study metric so that $\int_{\PP^n}\omega_{FS}^n=(2\pi)^n$.
\end{theorem}

This result gives a new characterization of the complex projective space in terms of Ricci and volume, and extends the previous works of Berman--Berndtsson \cite{BB}, F. Wang \cite{W} and K. Fujita \cite{F} (see also Y. Liu \cite{Liu18}) to general K\"ahler classes. As we shall see, while the statement of Theorem \ref{thm:vol-kahler} is differential geometric, its proof turns out to be 
rather algebraic.

Indeed, K\"ahler manifolds with positive Ricci curvature are automatically Fano. These are simply connected projective manifolds with many additional algebraic properties. So in what follows, unless otherwise specified, we will always assume that $X$ is an $n$-dimensional Fano manifold. Note that the Picard group $\operatorname{Pic}(X)\cong H^2(X,\ZZ)$ is a finitely generated torsion free Abelian group, hence a lattice. So the isomorphism classes of line bundles on $X$ are in one-to-one correspondence with the lattice points of $H^2(X,\ZZ)$. 
Also note that the \emph{K\"ahler cone} $\mathcal{K}(X)$ of $X$ coincides with its \emph{ample cone}, as $H^2(X,\RR)=H^{1,1}(X,\RR)$ by Kodaira vanishing and Hodge decomposition. So any K\"ahler class in $\mathcal{K}(X)$ can be approximated by a sequence of rational classes (corresponding to ample $\QQ$-line bundles).

Based on this, instead of using traditional Riemannian geometry, we will prove Theorem \ref{thm:vol-kahler} from an algebraic viewpoint, which requires several tools that have been developed in the K-stability theory. First of all, we will reformulate the curvature condition in terms of Tian's \emph{greatest Ricci lower bound}, which in turn can be related to the algebraic $\delta$-invariant (see \eqref{eq:def-delta-L}) thanks to the recent result of the author \cite[Appendix]{CRZ} and independently \cite[Theorem C]{BBJ18}. Then by modifying the argument of K. Fujita \cite{F}, we get the desired volume upper bound for K\"ahler classes. The essential difficulty of Theorem \ref{thm:vol-kahler} lies in the characterization of the equality. For this, we will compute the anti-canonical Seshadri constant as in \cite{F}. However, when dealing with general K\"ahler classes, the irrationality causes some subtle issue. To overcome this, we need some key observations from convex geometry. In particular, Newton--Okounkov bodies and the positivity criterion of K\"uronya--Lozovanu \cite{KL} will be crucially used to treat the equality case of Theorem \ref{thm:vol-kahler}.

Apart from the optimal volume upper bound, quantitative volume rigidity is also an important property in geometry.
Recall that, the classical sphere gap theorem in Riemannian geometry says the following:
\begin{theorem}\label{thm:cc}\cite[Theorem A.1.10]{CC97}.
There exists $\varepsilon(m)>0$, such that if $(M,g)$ is an $m$-dimensional Riemannian manifold with
$\Ric(g)\geq(m-1)g$ and $\vol(M,g)\geq\vol(S^m)-\varepsilon(m)$, then
$M$ is diffeomorphic to $S^m$.
\end{theorem}

So it is satisfactory to have the following K\"ahlerian analogue, which was stated as a conjecture in an earlier version of this paper. The proof is due to Yuchen Liu (in the toric setting, this was also obtained by F. Wang \cite{W} using combinatoric methods).

\begin{theorem}\label{thm:volgap}
There exists $\varepsilon(n)>0$, such that if $(X,\o)$ is an $n$-dimensional K\"ahler manifold with
$\Ric(\o)\geq(n+1)\o$ and $\vol(X,\omega)\geq\vol(\PP^n,\omega_{FS})-\varepsilon(n)$, then
$X$ is biholomorphic to $\PP^n$.
\end{theorem}

The rest of this paper is organized as follows. In Section \ref{sec:pre} we review some necessary notions and tools from the literature. In Section \ref{sec:rational-class} we prove Theorem \ref{thm:vol-kahler} by assuming that $[\o]$ is (a multiple of) a rational class. In Section \ref{sec:Kahler}, we prove Theorem \ref{thm:vol-kahler} in full generality. In the appendix provided by Y. Liu, Theorem \ref{thm:volgap} is proved.

\begin{remark}
After completing the first draft of this paper, the author was kindly informed by Feng Wang that he had also independently obtained the inequality in Theorem \ref{thm:vol-kahler} by adopting the argument of \cite{F} to a sequence of conic KE metrics.
\end{remark}

\textbf{Acknowledgments.}
The author would like to thank Kento Fujita, Feng Wang and Chuyu Zhou for many helpful discussions. He is also grateful to Yuchen Liu for providing the proof of Theorem \ref{thm:volgap}. Thanks also go to Xiaohua Zhu and Yanir Rubinstein for valuable comments. Special thanks go to the anonymous referees whose comments helped improve and clarify this manuscript.
The author is supported by the China post-doctoral grant BX20190014.

\section{Preliminaries}
\label{sec:pre}
In this section $X$ is assumed to be an $n$-dimensional Fano manifold.
\subsection{The volume function on the N\'eron--Severi space}
\hfill\\
Note that, $H^{1,1}(X,\RR)$ can be identified with the N\'eron--Severi space $N^1(X)_\RR$, which consists of numerical equivalence classes of $\RR$-divisors on $X$. One can define a \emph{continuous} volume function $\vol(\cdot)$ on $N^1(X)_\RR$. When restricted to the K\"ahler cone $\mathcal{K}(X)$ (i.e., the ample cone), $\vol(\cdot)$ is the usual volume for K\"ahler classes (which will be treated as ample $\RR$-divisors in what follows).

Also recall that, a class $\xi\in N^1(X)_\RR$ is called \emph{nef} if for every curve $C$ on $X$
$$
\xi\cdot C\geq 0.
$$
For nef classes $\xi$, $\vol(\xi)$ is simply equal to the top self-intersection number $\xi^n$. 

A class $\xi\in N^1(X)_\RR$  is called \emph{big} if
$$\vol(\xi)>0.$$
For more details on this subject, we refer the reader to the standard reference \cite{L04}.

\subsection{The greatest Ricci lower bound}
\hfill\\
Let $\mathcal{K}(X)$ denote the K\"ahler cone of $X$.
For any K\"ahler class $\xi\in\mathcal{K}(X)$, one can naturally define its
greatest Ricci lower bound $\beta(X,\xi)$ to be
\footnote{We put a factor $2\pi$ in the definition for convenience.}
\begin{equation}
    \label{eq:def-beta-xi}
    \beta(X,\xi):=\sup\{\mu>0\ |\ \exists\text{ K\"ahler form }\o\in 2\pi\xi\ \text{s.t. }\Ric(\o)\geq\mu\o \}.
\end{equation}

Note that, by the Calabi--Yau theorem, given any K\"ahler form $\alpha\in 2\pi c_1(X)$, one can always find $\o\in 2\pi\xi$ such that
$
\Ric(\o)=\alpha>0.
$
By compactness of $X$ we see
$
\Ric(\o)\geq\epsilon\o
$
for some $\epsilon>0$. So $\beta(X,\xi)$
is always a positive number. On the other hand, $\beta(X,\xi)$ is naturally bounded from above by the Seshadri constant
\begin{equation}
    \label{eq:def-sashadri-L}
    \epsilon(X,\xi):=\sup\{\mu>0\mid c_1(X)-\mu \xi\text{ is nef}\}.
\end{equation}
Thus we always have
\begin{equation}
\label{eq:beta<=epsilon}
    0<\beta(X,\xi)\leq \epsilon(X,\xi).
\end{equation}

When $\xi=c_1(L)$ for some ample $\QQ$-line bundle $L$, we will write
$$
\beta(X,L):=\beta(X,c_1(L))
$$
for ease of notation.

\begin{remark}
When $\xi= c_1(X)$, the greatest Ricci lower bound was first studied by Tian \cite{TianGreatesLowerRicci}, although it was not explicitly defined there. It was
first explicitly defined by Rubinstein in \cite{R08,R09},
and was later further studied by Sz\'ekelyhidi
\cite{GaborGreatesLowerRicci}, Li \cite{Litoric},
Song--Wang \cite{SongWang}, Cable \cite{Cable}, et al.
\end{remark}

\subsection{The $\delta$-invariant}
\hfill\\
Let $L$ be an ample $\QQ$-line bundle on $X$.
Following \cite{FO,BJ17}, the $\delta$-invariant of $L$ is defined by
\begin{equation}
    \label{eq:def-delta-L}
    \delta(X,L):=\inf_{E}\frac{A_X(E)}{S_L(E)}.
\end{equation}
Here $E$ runs through all the prime divisors \emph{over} $X$ (i.e., $E$ is a divisor contained in some birational model $Y\xrightarrow{\pi}X$ over $X$).
Moreover,
$$
A_X(E):=1+\ord_E(K_Y-\pi^*K_X),
$$
denotes the log discrepancy, and
$$
S_L(E):=\frac{1}{\vol(L)}\int_0^\infty\vol(\pi^*L-xE)dx$$ 
denotes the expected vanishing order of $L$ along $E$. 
Note that $\delta$-invariant is also called \emph{stability threshold} in the literature, which plays important roles in the study of K-stability and has attracted intensive research attentions. When $L=-K_X$, it was proved by the author in the appendix of his joint work with Cheltsov and Rubinstein \cite{CRZ} that
$$\beta(X,-K_X)=\min\{1,\delta(X,-K_X)\}.$$
For arbitrary ample $\QQ$-line bundles, we have the following independent result by Berman--Boucksom--Jonsson \cite{BBJ18}, giving a geometric interpretation of $\delta$-invariants on Fano manifolds.
\begin{theorem}
\label{thm:beta=delta-general}
Let $L$ be an ample $\QQ$-line bundle on a Fano manifold $X$. Then one has
$$
\beta(X,L)=\min\{\epsilon(X,L),\delta(X,L)\}.
$$
\end{theorem}

\subsection{Newton--Okounkov bodies and positivity of $\RR$-line bundles}
\hfill\\
We briefly recall the definition of Newton--Okounkov bodies; for more details we refer the reader to \cite{LM}.
Let $Y$ be an $n$-dimensional projective manifold.
Choose a flag of subvarieties
$$
Y_\bullet:\ Y=Y_0\supset Y_1\supset...\supset Y_{n-1}\supset Y_n=\{pt.\},
$$
such that each $Y_i$ is an irreducible subvariety of codimension $i$ and smooth at the point $Y_n$.
Such a flag is called \emph{admissible}.
Then any big class $\xi\in N^1(Y)_\RR$ can be associated with a convex body $\Delta_{Y_\bullet}(\xi)$ in $(\RR_{\geq0})^n$, which is called the Newton--Okounkov body of $\xi$ with respect to the flag $Y_\bullet$. This generalizes the classical polytope construction for divisors on toric varieties. 
A crucial fact is that
\begin{equation}
    \label{eq:vol=vol}
    \vol(\xi)=n!\vol_{\RR^n}(\Delta_{Y_\bullet}(\xi)).
\end{equation}
In this way one can study the volume function $\vol(\cdot)$ on $N^1(Y)_\RR$ using convex geometry. For more details of this construction we refer the reader to \cite{LM}.

It turns out that Newton--Okounkov bodies can also help us visualize the positivity of $\RR$-line bundles. More precisely, for any big $\RR$-divisor $\xi$, one can define its \emph{restricted base loci} by
$$
B_-(\xi):=\bigcup_A B(\xi+A),
$$
where the union is over all ample $\QQ$-divisors $A$ on Y and $B(\cdot)$ denotes the stable base loci (cf. \cite{ELMNP}). Then it is easy to see that
\begin{equation}
\xi\text{ is nef if and only if }B_-(\xi)=\emptyset.
\end{equation}
More precisely, $B_-(\xi)$ captures the non-nef locus of $\xi$ (see \cite[Example 1.18]{ELMNP}). Indeed, suppose that there exists some curve $C$ intersecting negatively with $\xi$, then by adding a small amount of ample $\QQ$-divisor A, one still has
$$
(\xi+A)\cdot C<0,
$$
which implies that $C\subset B(\xi+A)$ and hence
$$
C\subset B_-(\xi).
$$
The result of
K\"uronya--Lozovanu says that one can characterize the restricted base loci using Newton--Okounkov bodies.
\begin{theorem}\cite[Theorem A]{KL}
\label{thm:nef-criterion}
Let $\xi$ be a big $\RR$-divisor. Then the following are equivalent.
\begin{enumerate}
    \item $q\notin B_-(\xi)$.
    \item There exists an admissible flag $Y_\bullet$ with $Y_n=\{q\}$ such that the origin $0\in\Delta_{Y_\bullet}(\xi)\subset\RR^n$.
    \item For any admissible flag $Y_\bullet$ with $Y_n=\{q\}$, one has $0\in\Delta_{Y_\bullet}(\xi)\subset\RR^n$.
\end{enumerate}
\end{theorem}

Let us also record the following useful translation property of Newton--Okounkov bodies.
\begin{proposition}\cite[Proposition 1.6]{KL}.
\label{prop:translte-Delta}
Let $\xi$ be a big $\RR$-divisor and $Y_\bullet$ an admissible flag on $Y$. Then for any $t\in[0,\tau(\xi,Y_1))$ we have
$$
\Delta_{Y_\bullet}(\xi)_{\nu_1\geq t}=\Delta_{Y_\bullet}(\xi-tY_1)+te_1,
$$
where $\tau(\xi,Y_1):=\sup\{\mu>0|\xi-\mu Y_1\text{ is big}\}$ denotes the pseudo-effective threshold, $\nu_1$ denotes the first coordinate of $\RR^n$ and $e_1=(1,0,...,0)\in\RR^n$.
\end{proposition}

\section{Rational classes}
\label{sec:rational-class}

In this section, we will verify Theorem \ref{thm:vol-kahler} for rational classes. More precisely, we prove the following

\begin{theorem}
\label{thm:main-rational}
Let $L$ be an ample $\QQ$-line bundle on a Fano manifold $X$. Then one has
$$
\beta(X,L)^n\vol(L)\leq(n+1)^n,
$$
with equality if and only if $X$ is biholomorphic to $\PP^n$.
\end{theorem}

\begin{proof}
We follow the argument of K. Fujita \cite{F}.
Firstly, Theorem \ref{thm:beta=delta-general} implies that
$$
\delta(X,L)\geq\beta(X,L)\text{ and }\epsilon(X,L)\geq\beta(X,L).
$$
Pick any point $p\in X$ and let $\hat{X}\xrightarrow{\sigma}X$ be the blow-up at $p$. Let $E$ be the exceptional divisor of $\sigma$.
Then one has
$$
A_X(E)\geq \beta(X,L)S_L(E),
$$
and hence,
\begin{equation*}
    \begin{aligned}
    n=A_X(E)&\geq \beta(X,L)S_L(E)\\
    &=\frac{\beta(X,L)}{\vol(L)}\int_0^\infty\vol(\sigma^*L-xE)dx\\
    &\geq\frac{\beta(X,L)}{\vol(L)}\int_0^{\sqrt[n]{\vol(L)}}(\vol(L)-x^n)dx\\
    &=\frac{n\beta(X,L)}{n+1}\sqrt[n]{\vol(L)}.\\
    \end{aligned}
\end{equation*}
Here we used \cite[Theorem 2.3(1)]{F}. Thus
$$\vol(L)\leq\frac{(n+1)^n}{\beta(X,L)^n},$$
so the desired inequality is established.  
Now suppose that $\vol(L)=\frac{(n+1)^n}{\beta(X,L)^n}$. Then we see that the equality
$$
\vol(\sigma^*L-xE)=\vol(L)-x^n
$$
has to hold true for any $x\in[0,\frac{n+1}{\beta(X,L)}]$ (as $\vol(\sigma^*L-xE)$ is a continuous function in $x$). So \cite[Theorem 2.3(2)]{F} implies that
$$
\sigma^*L-\frac{n+1}{\beta(X,L)}E \text{ is nef.}
$$
Now using $\epsilon(X,L)\geq\beta(X,L)$, we find that
$$
\sigma^*(-K_X)-(n+1)E=\sigma^*\bigg(-K_X-\beta(X,L)L\bigg)+\beta(X,L)\bigg(\sigma^*L-\frac{n+1}{\beta(X,L)}E\bigg)
$$
is nef as well. Since $p\in X$ can be chosen arbitrarily, we conclude that $X\cong\PP^n$ by \cite{CMSB,K02}.
\end{proof}

In the above proof, we actually obtained the following general inequality (cf. Blum--Jonsson \cite[Theorem D]{BJ17})
\begin{equation}
    \label{eq:BJ17-thm-D}
    \delta(X,L)^n\vol(L)\leq (n+1)^n.
\end{equation}
This inequality reveals the deep relationship between \emph{singularities} and \emph{volumes} of linear systems. 

\begin{remark}
In the toric case when $L=-K_X$, the inequality
in Theorem \ref{thm:main-rational}
was first obtained by Berman--Berndtsson \cite{BB} using analytic methods, and the equality case was characterized by F. Wang \cite{W}.
\end{remark}

\section{General K\"ahler classes}
\label{sec:Kahler}

Let us attend to general K\"ahler classes.
The main result of this section is the following.

\begin{theorem}
\label{thm:main-general}
Let $\xi$ be a K\"ahler class of a Fano manifold $X$. Then one has
$$
\beta(X,\xi)^n\vol(\xi)\leq(n+1)^n,
$$
with equality if and only if $X$ is biholomorphic to $\PP^n$.
\end{theorem}

\begin{remark}
Using Bishop--Gromov, one can quickly derive
\begin{equation*}
\label{eq:beta-V-bd}
\beta(X,\xi)^n\vol(\xi)\leq\frac{2^{n+1}(n!)^2(2n-1)^n}{(2n)!}.
\end{equation*}
However this bound is much worse than $(n+1)^n$ (especially when $n$ is large).
\end{remark}

The proof of Theorem \ref{thm:main-general} will be divided into several steps. We first show the inequality by approximation and then characterize the equality using Newton--Okounkov bodies. We begin with a simple observation.

\begin{lemma}
\label{lem:lsc-beta}
The greatest Ricci lower bound $\beta(X,\cdot)$ is a lower semi-continuous
\footnote{In fact it is proved in the author's recent work \cite{Z20} that $\beta(X,\cdot)$ is even continuous.}
function on $\mathcal{K}(X)$.
\end{lemma}

\begin{proof}
Let $\{e_1,...,e_\rho\}$ be a basis of $H^{1,1}(X,\RR)$, then any K\"ahler class $\xi\in\mathcal{K}(X)$ can be
written as
$$
\xi=\sum_{i=1}^\rho a_ie_i
$$
for some $a_i\in\RR$.
For each $i\in\{1,...,\rho\}$ choose a smooth real $(1,1)$-form $\eta_i\in e_i$.  

Now assume that there exists $\o\in2\pi\xi$ such that
$$
\Ric(\o)>\mu\o
$$
for some $\mu>0$. For any $\vec{\epsilon}=(\epsilon_1,...,\epsilon_\rho)\in\RR^{^\rho}$ with $||\vec{\epsilon}||\ll1$, we put
$$
\o_{\vec{\epsilon}}:=\o+\sum_{i=1}^\rho\epsilon_i\eta_i.
$$
Then for $||\vec{\epsilon}||\ll1$ one also has
$$
\Ric(\o_{\vec{\epsilon}})>\mu\o_{\vec{\epsilon}}.
$$
So the lower semi-continuity of $\beta(X,\cdot)$ follows.

\end{proof}

As a consequence we get the following volume upper bound for general K\"ahler classes in terms of its greatest Ricci lower bound.

\begin{proposition}
\label{prop:beta-V-bd-for-xi}
Let $\xi$ be a K\"ahler class on an $n$-dimensional Fano manifold $X$. Then one has
\begin{equation}
    \label{eq:beta-omega-bd}
    \beta(X,\xi)^n\vol(\xi)\leq (n+1)^n.
\end{equation}
    
\end{proposition}

\begin{proof}
Choose a sequence of ample $\QQ$-line bundles $L_i$ such that
$$
L_i\rightarrow\xi \text{ in }N^1(X)_\RR.
$$
By Lemma \ref{lem:lsc-beta} we have
$$
\beta(X,\xi)\leq\liminf_i\beta(X,L_i).
$$
So for any $\epsilon>0$ and $i\gg1$, one has
$$
\beta(X,L_i)\geq \beta(X,\xi)-\epsilon.
$$
Thus Theorem \ref{thm:main-rational} implies that
$$
(\beta(X,\xi)-\epsilon)^n\vol(L_i)\leq(n+1)^n.
$$
Using the continuity $\vol(L_i)\rightarrow\vol(\xi)$ and sending $\epsilon\rightarrow0$, we get
$$
\beta(X,\xi)^n\vol(\xi)\leq (n+1)^n.
$$
\end{proof}

Therefore, to finish the proof of Theorem \ref{thm:main-general}, it remains to show that the equality of \eqref{eq:beta-omega-bd} is exactly obtained by $\PP^n$. Let us prepare the following lemma.

\begin{lemma}
\label{lem:Fujita-ineq-Kahler-class}
Let $X$ be a projective manifold. Pick any point $p\in X$ and let $\hat{X}\xrightarrow{\sigma}X$ be the blow-up at $p$. Let $E$ be the exceptional divisor of $\sigma$. Let $\xi\in N^1(X)_\RR$ be a nef and big $\RR$-line bundle. 
\begin{enumerate}
    \item For any $x\in\RR_{\geq0}$, one has
$$
\vol(\sigma^*\xi-xE)\geq\vol(\xi)-x^n.
$$
\item Suppose in addition that $X$ is Fano and that $\xi$ is ample satisfying
$$
\beta(X,\xi)^n\vol(\xi)=(n+1)^n,
$$
then for any $x\in[0,\vol(\xi)^{1/n}]$,
$$
\vol(\sigma^*\xi-xE)=\vol(\xi)-x^n.
$$
\end{enumerate}
\end{lemma}

\begin{proof}
This first part follows from \cite[Theorem 2.3(1)]{F} by approximation. Indeed, let $L_i$ be a sequence of ample $\QQ$-line bundles such that
$
L_i\rightarrow\xi
$ in $N^1(X)_\RR$.
Then for any $x\in\RR_{\geq0}$, \cite[Theorem 2.3(1)]{F} says that
$$
\vol(\sigma^*(L_i)-xE)\geq\vol(L_i)-x^n,
$$
so the assertion follows by the continuity of $\vol(\cdot)$.

For the second part, we rescale $\xi$ such that
$$
\beta(X,\xi)=1 \text{ and }\vol(\xi)=(n+1)^n.
$$
Let $L_i$ be a sequence of ample $\QQ$-line bundles such that
$
L_i\rightarrow\xi
$ in $N^1(X)_\RR$. For any $\epsilon>0$ and $i\gg1$, Theorem \ref{thm:beta=delta-general} and Lemma \ref{lem:lsc-beta} implies that
$$
\delta(X,L_i)\geq\beta(X,L_i)\geq 1-\epsilon.
$$
Thus we get
\begin{equation*}
    \begin{aligned}
    n=A_X(E)&\geq(1-\epsilon) S_{L_i}(E)\\
    &=\frac{1-\epsilon}{\vol(L_i)}\int_0^\infty\vol(\sigma^*L_i-xE)dx.\\
    \end{aligned}
\end{equation*}
Letting $i\rightarrow\infty$, by dominated convergence theorem and by sending $\epsilon\rightarrow 0$, we get
$$
n\geq\frac{1}{\vol(\xi)}\int_0^\infty\vol(\sigma^*\xi-xE)dx,
$$
so that (recall $\vol(\xi)=(n+1)^n$)
\begin{equation*}
    \begin{aligned}
    n&\geq\frac{1}{\vol(\xi)}\int_0^{\vol(\xi)^{1/n}}(\vol(\xi)-x^n)dx\\
    &=\frac{n}{n+1}\vol(\xi)^{1/n}=n.\\
    \end{aligned}
\end{equation*}
This gives
$$
\vol(\sigma^*\xi-xE)=\vol(\xi)-x^n
$$
for $x\in[0,\vol(\xi)^{1/n}]$ as claimed.
\end{proof}

As we have seen in the proof of Theorem \ref{thm:main-rational}, K. Fujita \cite[Theorem 2.3(2)]{F} says that, for ample $\QQ$-line bundles, the condition
$$
\vol(\sigma^*L-xE)=\vol(L)-x^n,\ x\in[0,a]
$$
implies that $\sigma^*L-xE$ is nef for $x\in[0,a]$, whose proof however heavily relies on the rationality of $L$ and the ampleness criterion of \cite{dFKL}.
To prove the same assertion for general ample $\RR$-line bundles, there is some subtlety involved if we follow Fujita's original argument. To overcome this, we take an alternative approach, using Newton--Okounkov bodies.

Before stating the key result, we recall the notion of \emph{local Seshadri constant}. Given a nef $\RR$-divisor $\xi$ and a point $p\in X$, let
$$
\epsilon_p:=\inf_{C}\frac{\xi\cdot C}{\operatorname{mult}_p(C)}
$$
denote the Seshadri constant of $\xi$ at $p$ (where the inf is over all the curves passing through $p$). Assume further that $\xi$ is big and nef, then $\epsilon_p>0$ for any general point $p\in X$. Indeed, as $\xi$ is big and nef, it can be written as $\xi=A+F$, where $A$ is an ample $\RR$-divisor and $F$ is an effective $\RR$-divisor (see \cite[Proposition 2.2.22]{L04}). Then $\xi$ has positive Seshadri constant at any point $p\in X-\operatorname{supp}(F)$.

\begin{proposition}
\label{prop:Fujita-nef-thm-for-R-divisor}
Let $X$ be a projective manifold. Let $\xi$ be a big and nef $\RR$-divisor. Pick a point $p\in X$ such that $\xi$ has positive Seshadri constant at $p$. Let $\hat{X}\xrightarrow{\sigma}X$ be the blow-up at $p$. Let $E$ be the exceptional divisor of $\sigma$. 
Suppose that there exists $a\in(0,\sqrt[n]{\vol(\xi)}]$ such that
$$
\vol(\sigma^*\xi-xE)=\vol(\xi)-x^n\text{ for any }x\in[0,a].
$$
Then $\sigma^*\xi-xE$ is nef for any $x\in[0,a]$.
\end{proposition}

\begin{proof}
Let
$\epsilon_p>0$
denote the Seshadri constant of $\xi$ at $p$. Then 
$$\sigma^*\xi-x E \text{ is nef for any } x\in[0,\epsilon_p].$$
Assume that $\epsilon_p<a$, otherwise we are done.

We argue by contradiction. Suppose that $\sigma^*\xi-x_0E$ is not nef for some $x_0\in(\epsilon_p,a)$.
Then there exists some curve $C$ intersecting $\sigma^*\xi-x_0E$ negatively. Note that such $C$ necessarily intersects $E$ by the projection formula.
Thus the restricted base loci $B_-(\sigma^*\xi-x_0E)$ intersects $E$ as well.
So we can pick a point
$$
q\in B_-(\sigma^*\xi-x_0E)\cap E.
$$
Moreover, we build an admissible flag $Y_\bullet$ on $\hat{X}$ such that
$$
Y_1:=E\text{ and }Y_n:=\{q\}.
$$
This is doable because $E\cong\PP^{n-1}$ and $Y_2\supset...\supset Y_n=\{q\}$ can be chosen to be a flag of linear subspaces in $\PP^{n-1}$.
Then we get a Newton--Okounkov body $\Delta_{Y_\bullet}(\sigma^*\xi)$. As $\sigma^*\xi$ is nef, $B_-(\sigma^*\xi)=\emptyset$. So 
Theorem \ref{thm:nef-criterion} implies that $$0\in\Delta_{Y_\bullet}(\sigma^*\xi).$$
Also note that, by our assumption,
$$
\vol(\sigma^*\xi-xE)=\vol(\xi)-x^n>0\ \text{for all } x\in[0,a).
$$
So $\sigma^*\xi-xE$ is big for all $x\in[0,a)$ and hence the pseudo-effective threshold $\tau(\xi,E)$ satisfies $\tau(\xi,E)\geq a$.
Now by Proposition \ref{prop:translte-Delta}, for any $x\in[0,a),$ the Newton--Okounkov body $\Delta_{Y\bullet}(\sigma^*\xi-xE)$ can be obtained from $\Delta_{Y_\bullet}(\sigma^*\xi)$ by truncating and translating in the first coordinate $\nu_1$ of $\RR^n$. Therefore (recall \eqref{eq:vol=vol}), for $x\in(0,a)$,
$$
\vol_{\RR^n}\bigg(\Delta_{Y_\bullet}(\sigma^*\xi)\cap\{0\leq \nu_1\leq x\}\bigg)=\frac{1}{n!}\bigg(\vol(\xi)-\vol(\sigma^*\xi-xE)\bigg)=\frac{x^n}{n!}.
$$
For $r\in(0,a)$, consider the slice
$$
S_r:=\Delta_{Y_\bullet}(\sigma^*\xi)\cap\{\nu_1=r\},
$$
and put
$$
A(r):=\vol_{\RR^{n-1}}(S_r).
$$
Then
$$
\int_0^xA(r)dr=\vol_{\RR^n}\bigg(\Delta_{Y_\bullet}(\sigma^*\xi)\cap\{0\leq \nu_1\leq x\}\bigg)=\frac{x^n}{n!},\ x\in(0,a),
$$
which implies that
$$
A(r)=\frac{r^{n-1}}{(n-1)!}\ \text{for }r\in(0,a).
$$
Note that the Brunn--Minkowski inequality in convex geometry says that $A(r)^{\frac{1}{n-1}}$ is concave in its support. However in our case, $A(r)^{\frac{1}{n-1}}$ is in fact \emph{linear}, so we are in the equality case of the Brunn--Minkowski inequality. This means that all the slices $S_r$ are \emph{homothetic}. 
We claim that, this forces 
$
\Delta_{Y_\bullet}(\sigma^*\xi)
$
to be a convex cone over the $(n-1)$-dimensional convex set $\Sigma:=\Delta_{Y_\bullet}(\sigma^*\xi)\cap\{v_1=a\}$. 
Indeed, for $r\in(0,a)$, consider the cone over $S_r$:
$$
C(S_r):=\{\lambda v\ |\ \lambda\in[0,1],\ v\in S_r\}.
$$
Then $C(S_r)\subseteq\Delta_{Y_\bullet}(\sigma^*\xi)\cap\{0\leq \nu_1\leq r\}$ by convexity. On the other hand,
$$
\vol_{\RR^n}(C(S_r))=\int_0^r(\frac{s}{r})^{n-1}A(r)ds=\frac{r^n}{n!}=\vol_{\RR^n}\bigg(\Delta_{Y_\bullet}(\sigma^*\xi)\cap\{0\leq \nu_1\leq r\}\bigg).
$$
This implies that
$$
\Delta_{Y_\bullet}(\sigma^*\xi)\cap\{0\leq \nu_1\leq r\}=C(S_r),\ \text{for any }r\in(0,a). 
$$
Sending $r\rightarrow a$, we find that
$
\Delta_{Y_\bullet}(\sigma^*\xi)\cap\{0\leq\nu_1\leq a\}
$ is a cone over $\Sigma$, as claimed. Now recall that $\sigma^*\xi-xE$ is nef for $x\in[0,\epsilon_p]$, so $B_-(\sigma^*\xi-xE)=\emptyset$ and hence by Theorem \ref{thm:nef-criterion},
$$
0\in\Delta_{Y_\bullet}(\sigma^*\xi-xE),\ \text{for all }x\in[0,\epsilon_p].
$$
Then by Proposition \ref{prop:translte-Delta}, $\Delta_{Y_\bullet}(\sigma^*\xi)\cap\{0\leq \nu_1\leq a\}$ contains the line segment $\{(x,0,...,0)|x\in[0,\epsilon_p]\}$,
so the cone property forces it to contain the \emph{whole} line segment
$$\{(x,0,...,0)|x\in[0,a]\}.$$
Intuitively, one has the following picture.
\begin{figure}[H]
\centering
\includegraphics[width=0.48\textwidth]{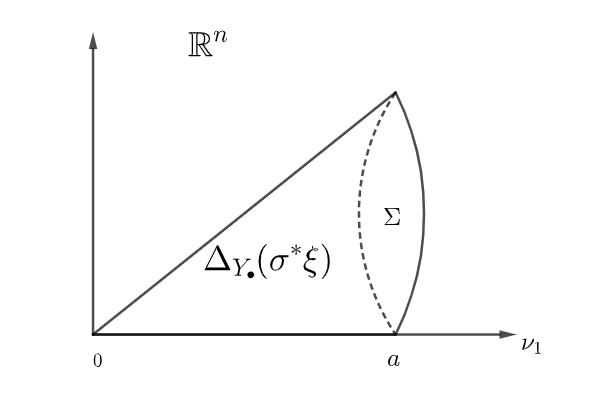}
\caption{$\Delta_{Y_\bullet}(\sigma^*\xi)\cap\{0\leq\nu_1\leq a\}$}
\end{figure}
Then by Proposition \ref{prop:translte-Delta} again,
$
0\in\Delta_{Y\bullet}(\sigma^*\xi-x_0E)
$
and hence $q\notin B_-(\sigma^*\xi-x_0E)$ by Theorem \ref{thm:nef-criterion}, which contradicts our choice of $q$.
Thus $\sigma^*\xi-xE$ is nef for any $x\in[0,a]$.

\end{proof}

\begin{remark}
We are kindly informed by a referee that, the above result also follows from \cite[Theorem 1.3]{PS}.
\end{remark}


\begin{proof}[Finishing the proof of Theorem \ref{thm:main-general}] It remains to consider the equality case: $\beta(X,\xi)^n\vol(\xi)=(n+1)^n$. Applying Lemma \ref{lem:Fujita-ineq-Kahler-class}(2) and Proposition \ref{prop:Fujita-nef-thm-for-R-divisor},
$$
\sigma^*\xi-\frac{n+1}{\beta(X,\xi)}E\text{ is nef}.
$$
On the other hand, by \eqref{eq:beta<=epsilon},
$$
-K_X-\beta(X,\xi)\xi\text{ is nef}.
$$
So we find that
$$
\sigma^*(-K_X)-(n+1)E=\sigma^*\bigg(-K_X-\beta(X,\xi)\xi\bigg)+\beta(X,\xi)\bigg(\sigma^*\xi-\frac{n+1}{\beta(X,\xi)}E\bigg)
$$
is nef as well. Since $p\in X$ can be chosen arbitrarily, we conclude that $X\cong\PP^n$ by \cite{CMSB,K02}.
\end{proof}

Finally, we are able to prove the main result of this paper.
\begin{theorem}
[=Theorem \ref{thm:vol-kahler}]
Let $(X,\o)$ be an $n$-dimensional K\"ahler manifold with
$$
\Ric(\o)\geq(n+1)\o.
$$
Then one has
$$
\int_X\o^n\leq(2\pi)^n,
$$
and the equality holds if and only if $(X,\o)$ is biholomorphically isometric to $(\PP^n,\o_{FS})$.
\end{theorem}

\begin{proof}
Consider the K\"ahler class $\xi:=\frac{1}{2\pi}[\omega]$. Then $\beta(X,\xi)\geq(n+1)$. So $\vol(\xi)\leq1$
by Proposition \ref{prop:beta-V-bd-for-xi}. 
In other words,
$$
\int_X\omega^n=(2\pi)^n\vol(\xi)\leq(2\pi)^n.
$$
And the equality holds if and only if $X\cong\PP^n$ by Theorem \ref{thm:main-general}, in which case, the equality
$
\int_{\PP^n}\o^n=(2\pi)^n
$
implies that
$$
[\o]=2\pi c_1(\mathcal{O}_{\PP^n}(1)).
$$
So $\ddbar$-lemma gives some $f\in C^\infty(\PP^n,\RR)$ such that
$$
\sqrt{-1}\ddbar f=\Ric(\o)-(n+1)\o\geq0,
$$
which forces $f$ to be a constant. Thus $\o$ satisfies the K\"ahler--Einstein equation
$$
\Ric(\o)=(n+1)\o.
$$
Now by the uniqueness of KE metrics \cite{BM}, we obtain $\o=\o_{FS}$ up to an automorphism.
\end{proof}

\newpage

\appendix

\section{Volume gap for K\"ahler manifolds with positive Ricci curvature}

  \begin{center}
    \author{by Yuchen Liu \\ Department of Mathematics, Yale University,\\
New Haven, CT 06511, USA\\
Email: \texttt{yuchen.liu@yale.edu}
}
\end{center}
\bigskip

In this appendix, we will prove Theorem \ref{thm:volgap}, which is a direct consequence of the following theorem.

\begin{theorem}\label{thm:gap}
Let $X$ be an $n$-dimensional Fano manifold. Let $\xi\in N^1(X)_{\bR}$ be an 
ample $\bR$-line bundle. Then there exists $\varepsilon=\varepsilon(n)>0$ such that
if 
\[
 \beta(X,\xi)^n\cdot\Vol(X,\xi)\geq (n+1)^n -\varepsilon,
\]
then $X$ is biholomorphic to $\bP^n$.
\end{theorem}


\begin{lemma}\label{lem:sesh}
Let $X$ be a Fano manifold. Let $\{L_i\}$ be a sequence 
of ample $\bQ$-line bundles on $X$. If $\epsilon(X,L_i)$ has
a positive lower bound, then $\{L_i\}$ is a bounded sequence in $N^1(X)_{\bR}$.
\end{lemma}

\begin{proof}
Choose a basis $\{[C_1], [C_2],\cdots, [C_\rho]\}$ of $N_1(X)_{\bR}$ where $\rho$ is the Picard rank of $X$ and each $C_j$ is an irreducible curve on $X$. Let  $||\cdot||$ be the norm on $N_1(X)_\RR$ as the maximum norm with respect to the basis $\{[C_j]\}$. We also denote its dual norm on $N^1(X)_\RR$ by $||\cdot||$ as abuse of notation. Choose $a>0$ such that  $\epsilon(X,L_i)\geq a>0$ for any $i$. Then $-K_X-aL_i$ is nef which implies
\[
0<(L_i\cdot C_j)\leq a^{-1}(-K_X\cdot C_j)
\]
for any $i$ and $j$. Hence we know that 
\[
||L_i||=\sup_{\max_{j}|c_j|=1} |(L_i\cdot \sum_{j=1}^\rho c_j C_j)|\leq \sum_{j=1}^\rho |(L_i\cdot C_j)|\leq a^{-1}\sum_{j=1}^\rho (-K_X\cdot C_j).
\]
Thus the proof is finished.
\end{proof}

\begin{proof}[Proof of Theorem \ref{thm:gap}]
 Assume to the contrary that $(X_i,\xi_i)$ is a sequence of Fano manifolds
 not biholomorphic to $\bP^n$ with ample $\bR$-line bundles,
 such that 
 \[
  \lim_{i\to\infty} \beta(X_i,\xi_i)^n\cdot\Vol(X_i,\xi_i)=(n+1)^n.
 \]
 Since $\beta(X,\cdot)$ is a lower semi-continuous function, 
 after perturbing $\xi_i$ to ample $\bQ$-line bundles $L_i$, we may assume that
 \begin{equation}\label{eq:appendix-0}
  \lim_{i\to\infty} \beta(X_i,L_i)^n\cdot\Vol(X_i,L_i)=(n+1)^n.
 \end{equation}
 By boundedness of Fano manifolds \cite{C92, KMM}, after passing to a subsequence 
 we may assume that there exists a smooth Fano family $\pi: \cX\to T$ 
 over an irreducible smooth base $T$ and a sequence of closed points $\{t_i\}\subset T$ such that
 $X_i\cong \cX_{t_i}$. By Lefschetz $(1,1)$ theorem and Kodaira vanishing theorem we know that $\Pic(X)\cong H^2(X,\bZ)$
 for a Fano manifold $X$. Hence after replacing $T$ by an irreducible component of its \'etale cover,
 we may assume that $R^2\pi_*\bZ$ is a trivial local system on $T$.
 Then $L_i$ extends to a $\bQ$-line bundle $\cL_i$ on $\cX$ that is generically $\pi$-ample. Since the volume of an ample $\bQ$-line bundle is the top self intersection number, we know that $\Vol(X_i, L_i)= \Vol(\cX_t, \cL_{i,t})$ for a general $t\in T$.
 
 Let $t\in T$ be a very general closed point. By the lower semi-continuity of $\delta$-invariants \cite{BL18} (with respect to the Zariski topology) and the genericity of ampleness, we know that 
 \[
 \delta(\cX_t, \cL_{i,t})\geq \delta(X_i, L_i)\quad \textrm{and}\quad \epsilon(\cX_t, \cL_{i,t})\geq \epsilon(X_i, L_i).
 \]
 Therefore, by Theorem \ref{thm:beta=delta-general} we have $\beta(\cX_t,\cL_{i,t})\geq \beta(X_i, L_i)$. This together with \eqref{eq:appendix-0} implies that 
 \begin{equation}\label{eq:appendix-1}
  \liminf_{i\to\infty} \beta(\cX_t,\cL_{i,t})^n\cdot\Vol(\cX_t,\cL_{i,t})\geq (n+1)^n.
 \end{equation}
 Let us choose a sequence of positive rational numbers $\{b_i\}_{i\in \bZ_{>0}}$ such that $
 \lim_{i\to\infty}b_i\cdot  \Vol(\cX_t, \cL_{i,t})^{\frac{1}{n}}=n+1$.
Denote by $\cL_i':= b_i \cL_i$ the rescaled $\bQ$-line bundle of $\cL_i$. Then we have
 \begin{equation}\label{eq:appendix-2}
 \lim_{i\to\infty}\Vol(\cX_t, \cL'_{i,t})=\lim_{i\to\infty}b_i^n\cdot \Vol(\cX_t, \cL_{i,t})=(n+1)^n.
 \end{equation}
 Since $\beta(\cX_t,\cL'_{i,t})=b_i^{-1}\cdot \beta(\cX_t,\cL_{i,t})$, \eqref{eq:appendix-1} and \eqref{eq:appendix-2} imply that 
 \begin{equation}\label{eq:appendix-3}
 \liminf_{i\to\infty} \beta(\cX_t,\cL'_{i,t})\geq 1.
 \end{equation}
 On the other hand, \eqref{eq:appendix-2} together with \eqref{eq:BJ17-thm-D} implies that 
 \begin{equation}\label{eq:appendix-4}
 \limsup_{i\to \infty} \delta(\cX_t, \cL'_{i,t})\leq 1.
 \end{equation}
 By Theorem \ref{thm:beta=delta-general}, we have $\beta(\cX_t, \cL'_{i,t})=\min\{\delta(\cX_t, \cL'_{i,t}),\epsilon(\cX_t, \cL'_{i,t})\}$. 
 Hence \eqref{eq:appendix-3} and \eqref{eq:appendix-4} imply that
 \begin{equation}\label{eq:appendix-5}
 \lim_{i\to\infty} \delta(\cX_t,\cL'_{i,t})=1 \quad\textrm{and}
 \quad \liminf_{i\to\infty}\epsilon(\cX_t,\cL'_{i,t})\geq 1.
 \end{equation}
 Hence by Lemma \ref{lem:sesh} we know that $\{\cL'_{i,t}\}$ is a bounded sequence of $N^1(\cX_t)_{\bR}$,
 which (after passing to a subsequence) converges to $\xi_\infty\in N^1(\cX_t)_{\bR}$. Then $\xi_\infty$ is a nef and big $\RR$-divisor with
 $\Vol(\cX_t,\xi_\infty)=(n+1)^n$. Pick a point $p\in\mathcal{X}_t$ such that $\xi_\infty$ has positive Seshadri constant at $p$. Let $\sigma:\hat{\cX}_t\to \cX_t$ be the blow up of $\cX_t$ at $p$ with exceptional divisor $E$. Let $\varepsilon>0$ be a positive constant. Then we have $\delta(\cX_t, \cL'_{i,t})\geq 1-\varepsilon$ for $i\gg 1$. Hence we have
 \[
 n=A_{\cX_t}(E)\geq (1-\varepsilon)S_{\cL'_{i,t}}(E)=\frac{1-\varepsilon}{\Vol(\cX_t, \cL'_{i,t})}\int_0^\infty\vol(\sigma^*\cL'_{i,t}-xE)dx.
 \]
 Then by the dominated convergence theorem similar to the proof of Lemma \ref{lem:Fujita-ineq-Kahler-class}, we have 
 \[
 n\geq \frac{1}{\Vol(\cX_t, \xi_\infty)}\int_0^\infty \vol(\sigma^*\xi_\infty -xE)dx. 
 \]
 Since $\Vol(\cX_t, \xi_\infty)=(n+1)^n$, the proof of Lemma \ref{lem:Fujita-ineq-Kahler-class} proceeds to showing that $\vol(\sigma^*\xi_\infty-xE)=\vol(\xi_\infty)-x^n$ for any $x\in [0, \vol(\xi_\infty)^{1/n}]$. Thus Proposition \ref{prop:Fujita-nef-thm-for-R-divisor} implies that $\sigma^*\xi_\infty-(n+1)E$ is nef. By \eqref{eq:appendix-5} we know that $-K_{\cX_t}-\xi_\infty$ is nef, hence $\sigma^*(-K_{\cX_t})-(n+1)E$ is nef. In other words, the Seshadri constant of $-K_{\cX_t}$ at $p$ is at least $n+1$. Hence we conclude that $\cX_t\cong \bP^n$ by \cite{BS09, LZ}. Then $X_i\cong\bP^n$ by rigidity of $\bP^n$
 under smooth deformation \cite[Exercise V.1.11.12.2]{K96},
 which is a contradiction.
\end{proof}

\end{document}